\documentclass{amsart}
\usepackage{amsfonts}
\usepackage{hyperref}

\newtheorem{theorem}{Theorem}
\theoremstyle{plain}

\newtheorem{corollary}{Corollary}

\newtheorem{lemma}{Lemma}

\newtheorem{proposition}{Proposition}
\newtheorem{remark}{Remark}

\numberwithin{equation}{section}
\input{tcilatex}

\begin{document}
\title[New sharp Cusa--Huygens type inequalities]{New sharp Cusa--Huygens
type inequalities for trigonometric and hyperbolic functions}
\address{Power Supply Service Center, ZPEPC Electric Power Research
Institute, Hangzhou, Zhejiang, China, 310009}
\email{yzhkm@163.com}
\date{April 10, 2014}
\thanks{This paper is in final form and no version of it will be submitted
for publication elsewhere.}

\begin{abstract}
We prove that for $p\in (0,1]$, the double inequality%
\begin{equation*}
\tfrac{1}{3p^{2}}\cos px+1-\tfrac{1}{3p^{2}}<\frac{\sin x}{x}<\tfrac{1}{%
3q^{2}}\cos qx+1-\tfrac{1}{3q^{2}}
\end{equation*}%
holds for $x\in (0,\pi /2)$ if and only if $0<p\leq p_{0}\approx 0.77086$
and $\sqrt{15}/5=p_{1}\leq q\leq 1$. While its hyperbolic version holds for $%
x>0$ if and only if $0<p\leq p_{1}=\sqrt{15}/5$ and $q\geq 1$. As
applications, some more accurate estimates for certain mathematical
constants are derived, and some new and sharp inequalities for
Schwab-Borchardt mean\ and logarithmic means are established.
\end{abstract}

\author{}
\maketitle

\section{Introduction}

The Cusa and Huygens (see, e.g., \cite{Huygens}) states that for $x\in
\left( 0,\pi /2\right) $, the inequality%
\begin{equation}
\frac{\sin x}{x}<\frac{2+\cos x}{3}  \label{C-H}
\end{equation}%
holds true. Its version of hyperbolic functions refers to (see \cite%
{Neuman.13.4.2010}) the inequality%
\begin{equation}
\frac{\sinh x}{x}<\frac{2+\cosh x}{3}  \label{Cusah}
\end{equation}%
holds for $x>0$, and it is know as hyperbolic Cusa--Huygens inequality (see 
\cite{Neuman.13.4.2010}).

There are many improvements, refinements and generalizations of (\ref{C-H})
and (\ref{Cusah}), see \cite{Sandor.RGMIA.8(3)(2005)}, \cite%
{Zhu.CMA.58(2009)}, \cite{Chen.JIA.2011.136}, \cite{Mortitc.MIA.14(2011)}, 
\cite{Nueman.MIA.15.2.2012}, \cite{Yang.arXiv.1206.5502.2012}, \cite%
{Yang.JMI.7.4.2013}; \cite{Zhu.MIA.11.3.2008}, \cite{Zhu.AAA.2009}, \cite%
{Zhu.CMA.58.2009.a}, \cite{Zhu.JIA.2010.130821}, \cite%
{Sandor.arXiv.1105.0859v1.2011}, \cite{Neuman.AIA.1.1.2012}, \cite%
{Wu.AML.25.5.2012}, \cite{Neuman.AMC.218.2012}, \cite{Yang.JIA.2013.116}, 
\cite{Chen.JMI.8.1.2014}, \cite{Yang.JIA.2013.541}.

Now we focus on the bounds for $\left( \sin x\right) /x$ in terms of $\cos
px $, where $x\in \left( 0,\pi /2\right) $, $p\in (0,1]$. In 1945, Iyengar 
\cite{Iyengar.6(1945)} (also see \cite[subsection 3.4.6]%
{Mitrinovic.Springer.1970} ) proved that for $x\in \left( 0,\pi /2\right) $, 
\begin{equation}
\cos px\leq \frac{\sin x}{x}\leq \cos qx  \label{1.1}
\end{equation}%
holds with the best possible constants%
\begin{equation*}
p=\frac{1}{\sqrt{3}}\text{ \ and \ \ }q=\frac{2}{\pi }\arccos \frac{2}{\pi }.
\end{equation*}%
Moreover, the following chain of inequalities hold:%
\begin{equation}
\cos x\leq \frac{\cos x}{1-x^{2}/3}\leq \left( \cos x\right) ^{1/3}\leq \cos 
\frac{x}{\sqrt{3}}\leq \frac{\sin x}{x}\leq \cos qx\leq \cos \frac{x}{2}\leq
1.  \label{1.2}
\end{equation}%
Qi et al. \cite{Qi.MIA.2(4)(1999)} showed that 
\begin{equation}
\cos ^{2}\frac{x}{2}<\frac{\sin x}{x}  \label{Qi}
\end{equation}%
holds for $x\in \left( 0,\pi /2\right) $. Kl\'{e}n et al. \cite[Theorem 2.4]%
{Klen.JIA.2010} pointed out that the function $p\mapsto \left( \cos
px\right) ^{1/p}$ is decreasing on $\left( 0,1\right) $ and for $x\in \left(
-\sqrt{27/5},\sqrt{27/5}\right) $%
\begin{equation}
\cos ^{2}\frac{x}{2}\leq \frac{\sin x}{x}\leq \cos ^{3}\frac{x}{3}\leq \frac{%
2+\cos x}{3}  \label{Klen}
\end{equation}%
are valid. Subsequently, Yang \cite{Yang.arXiv.1206.5502.2012} (also see 
\cite{Yang.GJM.2.1.2014}) gave a refinement of (\ref{Klen}), which states
that for $p,q\in \left( 0,1\right) $ the\ double inequality 
\begin{equation}
\left( \cos px\right) ^{1/p}<\frac{\sin x}{x}<\left( \cos qx\right) ^{1/q}
\label{Yang1}
\end{equation}%
holds for $x\in (0,\pi /2)$ if and only if $p\in \lbrack p_{0}^{\ast },1)$
and $q\in (0,1/3]$, where $p_{0}^{\ast }\approx 0.3473$. Moreover, the
double inequality%
\begin{equation*}
\left( \cos \frac{x}{3}\right) ^{\alpha }<\frac{\sin x}{x}<\left( \cos \frac{%
x}{3}\right) ^{3}
\end{equation*}%
with the best exponents $\alpha =2\left( \ln \pi -\ln 2\right) /\left( \ln
4-\ln 3\right) \approx 3.1395$ and $3$. Also, he pointed out that the value
range of variable $x$ such that (\ref{Klen}) holds can be extended to $%
\left( 0,\pi \right) $. Very recently, Yang \cite{Yang.JIA.2013.541} gave
another improvement of (\ref{Klen}), that is, for $x\in \left( 0,\pi
/2\right) $ the inequalities%
\begin{equation}
\frac{\sin x}{x}<\left( \tfrac{2}{3}\cos \tfrac{x}{2}+\tfrac{1}{3}\right)
^{2}<\cos ^{3}\frac{x}{3}<\frac{2+\cos x}{3}  \label{Yang2}
\end{equation}%
are true.

An important improvement for the inequality in (\ref{Qi}) is due to Neuman 
\cite{Neuman.13.4.2010}: 
\begin{equation}
\cos ^{4/3}\frac{x}{2}=\left( \frac{1+\cos x}{2}\right) ^{2/3}<\frac{\sin x}{%
x},\text{ }x\in \left( 0,\tfrac{\pi }{2}\right) .  \label{Neuman}
\end{equation}%
Lv et al. \cite{Lv.25(2012)} showed that for $x\in \left( 0,\pi /2\right) $\
inequalities 
\begin{equation}
\left( \cos \frac{x}{2}\right) ^{4/3}<\frac{\sin x}{x}<\left( \cos \frac{x}{2%
}\right) ^{\theta }  \label{Lv}
\end{equation}%
hold, where $\theta =2\left( \ln \pi -\ln 2\right) /\ln 2=1.3030...$\ and $%
4/3$\ are the best possible constants. By constructing a decreasing function 
$p\mapsto \left( \cos px\right) ^{1/\left( 3p^{2}\right) }$ ($p\in (0,1]$),
Yang \cite{Yang.JMI.7.4.2013} showed that the double inequality%
\begin{equation}
\left( \cos p_{1}^{\ast }x\right) ^{1/\left( 3p_{1}^{\ast 2}\right) }<\frac{%
\sin x}{x}<\cos ^{5/3}\frac{x}{\sqrt{5}}  \label{Yang4}
\end{equation}%
is true for $x\in \left( 0,\pi /2\right) $ with the best constants $%
p_{1}^{\ast }\approx 0.45346$ and $1/\sqrt{5}\approx 0.44721$. It follows
that%
\begin{eqnarray}
\left( \cos x\right) ^{1/3} &<&\cos ^{1/2}\tfrac{\sqrt{6}x}{3}<\cos ^{2/3}%
\tfrac{x}{\sqrt{2}}<\cos \frac{x}{\sqrt{3}}<\cos ^{4/3}\tfrac{x}{2}<\left(
\cos p_{1}^{\ast }x\right) ^{1/\left( 3p_{1}^{\ast 2}\right) }  \notag \\
\frac{\sin x}{x} &<&\cos ^{5/3}\tfrac{x}{\sqrt{5}}<\cos ^{2}\tfrac{x}{\sqrt{6%
}}<\cos ^{3}\tfrac{x}{3}<\cos ^{16/3}\tfrac{x}{4}<e^{-x^{2}/6}<\tfrac{2+\cos
x}{3}  \label{Yang5}
\end{eqnarray}%
are valid for $x\in \left( 0,\pi /2\right) $.

For the bounds for $\left( \sinh x\right) /x$ in terms of $\cosh px$, it is
known that the inequalities%
\begin{equation*}
\frac{\sinh x}{x}<\cosh ^{3}\frac{x}{3}<\frac{2+\cosh x}{3}
\end{equation*}%
holds true for $x>0$ (see \cite{Neuman.AMC.218.2012}), which is exactly
derived by the inequalities for means%
\begin{equation*}
L<A_{1/3}<\frac{2G+A}{3}
\end{equation*}%
(see e.g. \cite{Lin.AMM.81.1974}, \cite{Carlson.AMM.79.1972}, \cite%
{Yang.BKMS.49.1.2012}), where $L$, $A_{p}$, $G$ and $A$ stand for the
logarithmic mean, power mean of order $p$, geometric mean and arithmetic
mean or positive numbers $a$ and $b$ defined by%
\begin{eqnarray*}
L &\equiv &L\left( a,b\right) =\frac{a-b}{\ln a-\ln b}\text{ if }a\neq b%
\text{ and }L\left( a,a\right) =a, \\
A_{p} &\equiv &A_{p}\left( a,b\right) =\left( \frac{a^{p}+b^{p}}{2}\right)
^{1/p}\text{ if }p\neq 0\text{ and }A=A_{0}\left( a,b\right) =\sqrt{ab},
\end{eqnarray*}%
$G=A_{0}$ and $A=A_{1}$, respectively. Zhu in \cite{Zhu.2009.JIA.379142}
proved that \emph{for }$p>1$\emph{\ or }$p\leq 8/15$\emph{, and }$x\in
\left( 0,\infty \right) $, the inequality%
\begin{equation*}
\left( \frac{\sinh x}{x}\right) ^{q}>p+\left( 1-p\right) \cosh x
\end{equation*}%
is true if and only if $q\geq 3\left( 1-p\right) $. It follows by letting $%
p=1/2$ and $q=3/2$ that%
\begin{equation*}
\frac{\sinh x}{x}>\cosh ^{4/3}\frac{x}{2}
\end{equation*}%
holds for $x>0$ (also see \cite[(2.8)]{Neuman.13.4.2010}). Yang \cite%
{Yang.JIA.2013.116} showed that the inequality 
\begin{equation}
\frac{\sinh x}{x}>\left( \cosh px\right) ^{1/\left( 3p^{2}\right) }
\label{Yang6}
\end{equation}%
holds for all $x>0$\ if and only if$\ p\geq 1/\sqrt{5}$\ and its reverse
holds if and only if $0<p\leq 1/3$\emph{. }And, the function $p\mapsto
\left( \cosh px\right) ^{1/\left( 3p^{2}\right) }$\ is decreasing on $\left(
0,\infty \right) $.

The aim of this paper is to determine the best $p$ such that the inequalities%
\begin{eqnarray*}
\frac{\sin x}{x} &<&\left( >\right) \tfrac{1}{3p^{2}}\cos px+1-\tfrac{1}{%
3p^{2}}\text{, \ }p\in \left( 0,1\right) \text{, \ }x\in \left( 0,\pi
/2\right) , \\
\frac{\sinh x}{x} &<&\left( >\right) \tfrac{1}{3p^{2}}\cosh px+1-\tfrac{1}{%
3p^{2}}\text{, \ }p,x\in \left( 0,\infty \right)
\end{eqnarray*}%
hold true.

Our main results are contained in the following theorems.

\begin{theorem}
\label{MTt}For $p\in (0,1]$ and $x\in (0,\pi /2)$, the double inequality%
\begin{equation}
\tfrac{1}{3p^{2}}\cos px+1-\tfrac{1}{3p^{2}}<\frac{\sin x}{x}<\tfrac{1}{%
3q^{2}}\cos qx+1-\tfrac{1}{3q^{2}}  \label{M1}
\end{equation}%
holds if and only if $0<p\leq p_{0}\approx 0.77086$ and $0.77460\approx 
\sqrt{15}/5=p_{1}\leq q\leq 1$, where $p_{0}$ is the unique root of the
equation%
\begin{equation}
F_{p}\left( \tfrac{\pi }{2}^{-}\right) =\frac{2}{\pi }-\left( \tfrac{1}{%
3p^{2}}\cos \tfrac{p\pi }{2}+1-\tfrac{1}{3p^{2}}\right) =0  \label{F_p=0}
\end{equation}%
on $(0,1)$. And, the bound for $\left( \sin x\right) /x$ given in (\ref{M1})
is increasing with respect to parameters $p$ or $q$.
\end{theorem}

\begin{theorem}
\label{MTh}For $p,x>0$, the double inequality%
\begin{equation}
\tfrac{1}{3p^{2}}\cosh px+1-\tfrac{1}{3p^{2}}<\frac{\sinh x}{x}<\tfrac{1}{%
3q^{2}}\cosh qx+1-\tfrac{1}{3q^{2}}  \label{M2}
\end{equation}%
holds if and only if $0<p\leq p_{1}=\sqrt{15}/5$ and $q\geq 1$. And, the
bound for $\left( \sinh x\right) /x$ given in (\ref{M2}) is increasing with
respect to parameters $p$ or $q$.
\end{theorem}

\begin{remark}
The weighted basic inequality of two positive numbers of $a$ and $b$ tell us
that for $\alpha \in \lbrack 0,1]$, the inequality $\alpha a+\left( 1-\alpha
\right) b\geq a^{\alpha }b^{1-\alpha }$. It is reversed if and only if $%
\alpha \geq 1$ or $\alpha \leq 0$ (see \cite{Yang.CM.23.5.2007}). Hence,
taking into account (\ref{Yang4}) and (\ref{M1}) we see that

(i) if $p\in \lbrack p_{1}^{\ast },1/\sqrt{3})$, where $p_{1}^{\ast }\approx
0.45346$, then%
\begin{equation*}
\frac{\sin x}{x}>\left( \cos px\right) ^{1/\left( 3p^{2}\right) }>\tfrac{1}{%
3p^{2}}\cos px+1-\tfrac{1}{3p^{2}};
\end{equation*}%
(ii) if $p\in (1/\sqrt{3},p_{0})$, where $p_{0}\approx 0.77086$, then%
\begin{equation*}
\frac{\sin x}{x}>\tfrac{1}{3p^{2}}\cos px+1-\tfrac{1}{3p^{2}}>\left( \cos
px\right) ^{1/\left( 3p^{2}\right) }.
\end{equation*}%
In the same way, (\ref{Yang6}) together with (\ref{M2}) leads us to

(iii) if $p\in \lbrack 1/\sqrt{5},1/\sqrt{3})$, then 
\begin{equation*}
\frac{\sinh x}{x}>\left( \cosh px\right) ^{1/\left( 3p^{2}\right) }>\tfrac{1%
}{3p^{2}}\cosh px+1-\tfrac{1}{3p^{2}};
\end{equation*}%
(iv) if $p\in (1/\sqrt{3},\sqrt{15}/5)$, then 
\begin{equation*}
\frac{\sinh x}{x}>\tfrac{1}{3p^{2}}\cosh px+1-\tfrac{1}{3p^{2}}>\left( \cosh
px\right) ^{1/\left( 3p^{2}\right) }.
\end{equation*}
\end{remark}

Taking $p=3/4,1/\sqrt{2},2/3,1/\sqrt{3}$ and $q=\sqrt{3/5},\sqrt{2/3},\sqrt{3%
}/2,1$ in Theorem \ref{MTt}, we have

\begin{corollary}
For $x\in (0,\pi /2)$, the inequalities%
\begin{eqnarray}
\cos \tfrac{x}{\sqrt{3}} &<&\tfrac{3}{4}\cos \tfrac{2x}{3}+\tfrac{1}{4}<%
\tfrac{2}{3}\cos \tfrac{x}{\sqrt{2}}+\tfrac{1}{3}<\tfrac{16}{27}\cos \tfrac{%
3x}{4}+\tfrac{11}{27}<\frac{\sin x}{x}  \label{M1c} \\
&<&\tfrac{5}{9}\cos \tfrac{\sqrt{15}x}{5}+\tfrac{4}{9}<\cos ^{2}\tfrac{x}{%
\sqrt{6}}<\tfrac{4}{9}\cos \tfrac{\sqrt{3}x}{2}+\tfrac{5}{9}<\tfrac{1}{3}%
\cos x+\tfrac{2}{3}.  \notag
\end{eqnarray}
\end{corollary}

Putting $p=\sqrt{3/5},3/4,1/\sqrt{2},2/3,1/\sqrt{3}$ and $q=1,2/\sqrt{3}$ in
Theorem \ref{MTh} we have

\begin{corollary}
\label{MThc}For $x>0$, the inequalities%
\begin{eqnarray}
\cosh \tfrac{x}{\sqrt{3}} &<&\tfrac{3}{4}\cosh \tfrac{2x}{3}+\tfrac{1}{4}<%
\tfrac{2}{3}\cosh \tfrac{x}{\sqrt{2}}+\tfrac{1}{3}<\tfrac{16}{27}\cosh 
\tfrac{3x}{4}+\tfrac{11}{27}  \label{M2c} \\
&<&\tfrac{5}{9}\cosh \tfrac{\sqrt{15}x}{5}+\tfrac{4}{9}<\frac{\sinh x}{x}<%
\tfrac{1}{3}\cosh x+\tfrac{2}{3}<\tfrac{1}{2}\cosh ^{2}\tfrac{x}{\sqrt{3}}+%
\tfrac{1}{2}.  \notag
\end{eqnarray}
\end{corollary}

\section{Proof of Theorem 1}

In order to prove Theorem \ref{MTt}, we need some lemmas.

\begin{lemma}
\label{Lemma U_p}For $x\in (0,\pi /2)$, the function $p\mapsto U_{p}\left(
x\right) $ defined on $[0,1]$ by%
\begin{equation*}
U_{p}\left( x\right) =\frac{1}{3p^{2}}\cos px+1-\frac{1}{3p^{2}}\text{ if }%
p\in (0,1]\text{ and }U_{0}\left( x\right) =1-\frac{x^{2}}{6}
\end{equation*}%
is increasing.
\end{lemma}

\begin{proof}
Differentiation yields%
\begin{eqnarray*}
\frac{\partial U_{p}}{\partial p} &=&\frac{1}{3p^{3}}\left( 2-2\cos
px-px\sin px\right) \\
&=&\frac{12px}{3p^{3}}\left( \frac{\sin \frac{px}{2}}{\frac{px}{2}}-\cos 
\frac{px}{2}\right) \sin \frac{px}{2}>0,
\end{eqnarray*}%
which completes the proof.
\end{proof}

\begin{lemma}
\label{Lemma F_p}Let the function $F_{p}$ be defined on $(0,\pi /2)$ by%
\begin{equation}
F_{p}\left( x\right) =\tfrac{\sin x}{x}-(\tfrac{1}{3p^{2}}\cos px+1-\tfrac{1%
}{3p^{2}})\text{ if }p\in (0,1]\text{ and }F_{0}\left( x\right) =\tfrac{\sin
x}{x}-1+\tfrac{x^{2}}{6}.  \label{F_p}
\end{equation}

(i) If $F_{p}\left( x\right) <0$ for all $x\in (0,\pi /2)$, then $p\in
\lbrack p_{1},1]$, where $p_{1}=\sqrt{15}/5\approx 0.77460$.

(ii) If $F_{p}\left( x\right) >0$ for all $x\in (0,\pi /2)$, then $p\in
\lbrack 0,p_{0}]$, where $p_{0}\approx 0.77086$.
\end{lemma}

\begin{proof}
(i) If $F_{p}\left( x\right) <0$ for all $x\in (0,\pi /2)$, then we have%
\begin{equation}
\lim_{x\rightarrow 0}\frac{\frac{\sin x}{x}-\left( \frac{1}{3p^{2}}\cos px+1-%
\frac{1}{3p^{2}}\right) }{x^{4}}=-\frac{1}{360}\left( 5p^{2}-3\right) \leq 0,
\label{LimitF_p1}
\end{equation}%
which leads to $p\in (\sqrt{15}/5,1]$.

(ii) If $F_{p}\left( x\right) >0$ for all $x\in (0,\pi /2)$, then we have%
\begin{equation*}
F_{p}\left( \frac{\pi }{2}^{-}\right) =\frac{2}{\pi }-\left( \frac{1}{3p^{2}}%
\cos \frac{p\pi }{2}+1-\frac{1}{3p^{2}}\right) >0.
\end{equation*}%
From Lemma \ref{Lemma U_p} we see that the function $p\mapsto F_{p}\left(
\pi /2^{-}\right) $ is decreasing on $[0,1]$, which together with the facts%
\begin{equation*}
F_{1/2}\left( \frac{\pi }{2}^{-}\right) =\frac{2}{\pi }-\frac{2}{3}\sqrt{2}+%
\frac{1}{3}>0\text{ and }F_{1}\left( \frac{\pi }{2}^{-}\right) =\frac{2}{\pi 
}-\frac{2}{3}<0
\end{equation*}%
gives that there is a unique number $p_{0}\in (1/2,1)$ such that $%
F_{p}\left( \pi /2^{-}\right) >0$ for $p\in (0,p_{0})$ and $F_{p}\left( \pi
/2^{-}\right) <0$ for $p\in (p_{0},1)$. Solving the equation $F_{p}\left(
\pi /2^{-}\right) =0$ for $p$ by mathematical computer software we find that 
$p_{0}\approx 0.77086$.

This completes the proof.
\end{proof}

\begin{lemma}
\label{Lemma a_n}Let $c\in (0,3/5]$ and let the sequence $(a_{n}(c))$\ be
defined by 
\begin{equation}
a_{n}(c)=3-\left( 2n+1\right) c^{n-1}.  \label{a_n}
\end{equation}%
Then (i) $a_{n}(c)\geq 0$ for $n\in \mathbb{N}$; (ii) for $n\geq 3$, we have%
\begin{equation*}
1<\frac{a_{n+1}\left( c\right) }{a_{n}\left( c\right) }\leq \frac{%
a_{n+1}\left( 3/5\right) }{a_{n}\left( 3/5\right) }\leq \frac{11}{5}.
\end{equation*}
\end{lemma}

\begin{proof}
(i) We first show that $a_{n}(c)\geq 0$ for $n\in \mathbb{N}$. A simple
computation leads to%
\begin{eqnarray*}
a_{n+1}(c)-a_{n}(c) &=&\left( 2n+1\right) c^{n-1}-\left( 2n+3\right)
c^{n}=c^{n-1}\left( \left( 2n+1\right) -\left( 2n+3\right) c\right) \\
&\geq &c^{n-1}\left( \left( 2n+1\right) -\left( 2n+3\right) \frac{3}{5}%
\right) =\frac{4}{5}c^{n-1}\left( n-1\right) \geq 0,
\end{eqnarray*}%
which implies that $a_{n+1}(c)\geq a_{n}(c)\geq a_{1}(c)=0$.

(ii) Since $a_{1}(c)=0$, $a_{2}(c)=3-5c\geq 0$, $a_{n}(c)>0$ for $n\geq 3$,
if we can show that the function%
\begin{equation*}
\left( c,n\right) \mapsto \frac{a_{n+1}\left( c\right) }{a_{n}\left(
c\right) }=\frac{3-\left( 2n+3\right) c^{n}}{3-\left( 2n+1\right) c^{n-1}}
\end{equation*}%
is increasing in $c$ on $(0,3/5]$ and decreasing in $n\geq 3$, then we have%
\begin{equation*}
1=\frac{a_{n+1}\left( 0\right) }{a_{n}\left( 0\right) }<\frac{a_{n+1}\left(
c\right) }{a_{n}\left( c\right) }<\frac{a_{n+1}\left( 3/5\right) }{%
a_{n}\left( 3/5\right) }\leq \frac{a_{4}\left( 3/5\right) }{a_{3}\left(
3/5\right) }=\frac{11}{5},
\end{equation*}%
which is the desired results. Now we prove that $\left( c,n\right) \mapsto
a_{n+1}\left( c\right) /a_{n}\left( c\right) $ is increasing in $c$ on $%
(0,3/5]$ for $n\geq 3$. Differentiation yields%
\begin{equation*}
c^{2-n}a_{n}^{2}\left( c\right) \left( \frac{a_{n+1}\left( c\right) }{%
a_{n}\left( c\right) }\right) ^{\prime }=\left( 4n^{2}+8n+3\right)
c^{n}-3n\left( 2n+3\right) c+\left( 6n^{2}-3n-3\right) :=h_{n}\left(
c\right) ,
\end{equation*}%
\begin{equation*}
h_{n}^{\prime }(c)=-n\left( 2n+3\right) \left( 3-\left( 2n+1\right)
c^{n-1}\right) =-n\left( 2n+3\right) a_{n}\left( c\right) <0,
\end{equation*}%
where the last inequality holds due to $a_{n}(c)>0$ for $n\geq 3$.
Therefore, we have 
\begin{equation*}
h_{n}\left( c\right) \geq h_{n}\left( 3/5\right) =\left( \frac{3}{5}\right)
^{n}\left( 4n^{2}+8n+3\right) +\frac{3}{5}\left( 4n^{2}-14n-5\right) ,
\end{equation*}%
which is clearly positive due to that $h_{3}\left( 3/5\right) =876/125>0$
and $\left( 4n^{2}-14n-5\right) =n(4n-14)-5\geq 3$ for $n\geq 4$. This
reveals that $h_{n}^{\prime }(c)>0$, that is, $\left( c,n\right) \mapsto
a_{n+1}\left( c\right) /a_{n}\left( c\right) $ is increasing in $c$ on $%
(0,3/5]$.

On the other hand, we have%
\begin{eqnarray*}
\frac{a_{n+1}\left( c\right) }{a_{n}\left( c\right) }-\frac{a_{n+2}\left(
c\right) }{a_{n+1}\left( c\right) } &=&c^{n-1}\frac{4c^{n+1}+6\left(
c-1\right) ^{2}n+15c^{2}-18c+3}{a_{n}\left( c\right) a_{n+1}\left( c\right) }
\\
&\geq &\frac{c^{n-1}}{a_{n}\left( c\right) a_{n+1}\left( c\right) }\left(
4c^{n+1}+6\left( c-1\right) ^{2}\times 3+15c^{2}-18c+3\right) \\
&=&\frac{c^{n-1}}{a_{n}\left( c\right) a_{n+1}\left( c\right) }\left(
4c^{n+1}+33\left( \tfrac{7}{11}-c\right) \left( 1-c\right) \right) >0,
\end{eqnarray*}%
where the first inequality holds due to $a_{n}\left( c\right) >0$ for $n\geq
3$, while the last one holds since $c\in (0,3/5]$. This means that $\left(
c,n\right) \mapsto a_{n+1}\left( c\right) /a_{n}\left( c\right) $ is
decreasing with $n\geq 3$.

Thus we complete the proof of this assertion.
\end{proof}

Now we are in a position to prove Theorem \ref{MTt}.

\begin{proof}[Proof of Theorem \protect\ref{MTt}]
Expanding in power series gives%
\begin{eqnarray*}
F_{p}\left( x\right) &=&\tfrac{\sin x}{x}-(\tfrac{1}{3p^{2}}\cos px+1-\tfrac{%
1}{3p^{2}}) \\
&=&\sum_{n=0}^{\infty }\left( -1\right) ^{n}\frac{x^{2n}}{\left( 2n+1\right)
!}-\left( \frac{1}{3p^{2}}\sum_{n=0}^{\infty }\left( -1\right) ^{n}\frac{%
\left( px\right) ^{2n}}{\left( 2n\right) !}+1-\frac{1}{3p^{2}}\right) \\
&=&\sum_{n=2}^{\infty }\left( -1\right) ^{n}\frac{3-\left( 2n+1\right)
p^{2n-2}}{3\left( 2n+1\right) !}x^{2n} \\
&=&\frac{3-5p^{2}}{360}x^{4}+\sum_{n=3}^{\infty }\left( -1\right) ^{n}\frac{%
a_{n}\left( p^{2}\right) }{3\left( 2n+1\right) !}x^{2n},
\end{eqnarray*}%
where $a_{n}\left( c\right) $ is defined by (\ref{a_n}). Considering the
function $f_{p}\left( x\right) =x^{-4}F_{p}\left( x\right) $, we have%
\begin{equation}
f_{p}\left( x\right) =x^{-4}F_{p}\left( x\right) =\frac{3-5p^{2}}{360}%
+\sum_{n=3}^{\infty }\left( -1\right) ^{n}\frac{a_{n}\left( p^{2}\right) }{%
3\left( 2n+1\right) !}x^{2n-4},  \label{f_p}
\end{equation}%
and differentiation yields%
\begin{eqnarray*}
f_{p}^{\prime }\left( x\right) &=&\sum_{n=3}^{\infty }\left( -1\right) ^{n}%
\frac{\left( 2n-4\right) a_{n}\left( p^{2}\right) }{3\left( 2n+1\right) !}%
x^{2n-5} \\
&:&=\sum_{n=3}^{\infty }\left( -1\right) ^{n}u_{n}(x),
\end{eqnarray*}%
where%
\begin{equation*}
u_{n}(x)=\frac{\left( 2n-4\right) a_{n}\left( p^{2}\right) }{3\left(
2n+1\right) !}x^{2n-5}.
\end{equation*}%
Utilizing Lemma \ref{Lemma a_n} we get that for $p^{2}\in (0,3/5]$ and $%
n\geq 3$,%
\begin{eqnarray*}
\frac{u_{n+1}(x)}{u_{n}(x)} &=&\frac{\frac{\left( 2n-2\right) a_{n+1}\left(
p^{2}\right) }{3\left( 2n+3\right) !}x^{2n-3}}{\frac{\left( 2n-4\right)
a_{n}\left( p^{2}\right) }{3\left( 2n+1\right) !}x^{2n-5}}=\frac{1}{\left(
2n-4\right) \left( 2n+3\right) }\frac{\left( 2n-2\right) }{\left(
2n+2\right) }\frac{a_{n+1}\left( p^{2}\right) }{a_{n}\left( p^{2}\right) }%
x^{2} \\
&<&\frac{1}{\left( 2\times 3-4\right) \left( 2\times 3+3\right) }\times
1\times \frac{11}{5}\times \frac{\pi ^{2}}{4}=\frac{11\pi ^{2}}{360}<1,
\end{eqnarray*}%
which implies that the power series $\sum_{n=3}^{\infty }\left( -1\right)
^{n}u_{n}(x)$ is a Leibniz type alternating one, and so $f_{p}^{\prime
}\left( x\right) <0$ for $p^{2}\in (0,3/5]$.

(i) We first prove the second inequality in (\ref{M1}) holds, where $p_{1}=%
\sqrt{15}/5$ is the best. As shown previously, we see that $f_{p_{1}}$ is
decreasing on $(0,\pi /2)$, and therefore,%
\begin{equation*}
f_{p_{1}}\left( x\right) <f_{p_{1}}\left( 0^{+}\right) =\lim_{x\rightarrow
0^{+}}\left( x^{-4}F_{p}\left( x\right) \right) =\frac{1}{360}\left(
3-5p_{1}^{2}\right) =0,
\end{equation*}%
which together with (\ref{f_p}) yields $F_{p_{1}}\left( x\right) <0$ for $%
x\in \left( 0,\pi /2\right) $.

Next we prove $p_{1}=\sqrt{15}/5$ is the best. If there is another $%
p_{1}^{\ast }<p_{1}$ such that\ the second inequality in (\ref{M1}) holds
for $x\in (0,\pi /2)$, then by Lemma \ref{Lemma F_p} there must be $%
p_{1}^{\ast }\in \lbrack p_{1},1]$, which yields a contradiction. Therefore, 
$p_{1}=\sqrt{15}/5$ can not be replaced with other smaller ones.

(ii) Now we prove the first inequality in (\ref{M1}) holds with the best
constant $p_{0}\approx 0.77088$. Since $p_{0}^{2}\in (0,3/5]$, $f_{p_{0}}$
is also decreasing on $(0,\pi /2)$, and so%
\begin{equation*}
f_{p_{0}}\left( x\right) >f_{p_{0}}\left( \frac{\pi }{2}^{-}\right)
=\lim_{x\rightarrow \pi /2^{-}}\left( x^{-4}F_{p_{0}}\left( x\right) \right)
=\left( \frac{\pi }{2}\right) ^{-4}F_{p_{0}}\left( \frac{\pi }{2}^{-}\right)
=0,
\end{equation*}%
where the last equality is true due to $p_{0}$ is the unique root of the
equation (\ref{F_p=0}) on $\left( 0,1\right) $. It together with (\ref{f_p})
gives $F_{p_{0}}\left( x\right) >0$ for $x\in \left( 0,\pi /2\right) $.

Lastly, we show that $p_{0}$ is the best. Assume that there is another $%
p_{0}^{\ast }>p_{0}$ such that $F_{p_{0}^{\ast }}\left( x\right) >0$ for $%
x\in (0,\pi /2)$. Then by Lemma \ref{Lemma F_p} there must be $p_{0}^{\ast
}\in \lbrack 0,p_{0}]$, which is clear a contradiction. Consequently, $p_{0}$
can not be replaced by other larger numbers.

Thus the proof is complete.
\end{proof}

\begin{remark}
Application of the conclusion that $f_{p}^{\prime }\left( x\right) <0$ for $%
x\in \left( 0,\pi /2\right) $ if $p^{2}\in (0,3/5]$ gives $f_{p}\left( \pi
/2^{-}\right) <f_{p}\left( x\right) <f_{p}\left( 0^{+}\right) $, that is,%
\begin{equation*}
\left( \frac{\pi }{2}\right) ^{-4}F_{p}\left( \frac{\pi }{2}^{-}\right)
<x^{-4}F_{p}\left( x\right) <\lim_{x\rightarrow 0^{+}}x^{-4}F_{p}\left(
x\right) =\frac{3-5p^{2}}{360},
\end{equation*}%
which can be changed into%
\begin{equation}
(\tfrac{1}{3p^{2}}\cos px+1-\tfrac{1}{3p^{2}})+c_{0}\left( p\right) x^{4}<%
\frac{\sin x}{x}<(\tfrac{1}{3p^{2}}\cos px+1-\tfrac{1}{3p^{2}})+c_{1}\left(
p\right) x^{4},  \label{sin E}
\end{equation}%
where $c_{0}\left( p\right) =\left( \pi /2\right) ^{-4}F_{p}\left( \pi
/2^{-}\right) $ and $c_{1}\left( p\right) =\left( 3-5p^{2}\right) /360$ are
the best constants. Then

(i) when $p=p_{1}=\sqrt{15}/5$, we have%
\begin{equation*}
c_{0}\left( p_{1}\right) x^{4}+(\tfrac{1}{3p_{1}^{2}}\cos p_{1}x+1-\tfrac{1}{%
3p_{1}^{2}})<\frac{\sin x}{x}<(\tfrac{1}{3p_{1}^{2}}\cos p_{1}x+1-\tfrac{1}{%
3p_{1}^{2}}),
\end{equation*}%
where $c_{0}\left( p_{1}\right) =\left( \pi /2^{-}\right)
^{-4}F_{p_{1}}\left( \pi /2^{-}\right) \approx -7.2618\times 10^{-5}$ and $%
c_{1}\left( p_{1}\right) =0$ are the best possible constants;

(ii) when $p=p_{0}\approx 0.77086$, we get%
\begin{equation*}
(\tfrac{1}{3p_{0}^{2}}\cos p_{0}x+1-\tfrac{1}{3p_{0}^{2}})<\tfrac{\sin x}{x}%
<(\tfrac{1}{3p_{0}^{2}}\cos p_{0}x+1-\tfrac{1}{3p_{0}^{2}})+c_{1}\left(
p_{0}\right) x^{4},
\end{equation*}%
where $c_{0}\left( p_{0}\right) =0$ and $c_{1}\left( p_{0}\right) =\left(
3-5p_{0}^{2}\right) /360\approx 8.0206\times 10^{-5}$ are the best constants.
\end{remark}

\section{Proof of Theorem 2}

For proving Theorem \ref{MTh}, we first give the following lemmas.

\begin{lemma}
\label{Lemma V_p}For $x\in (0,\infty )$, the function $p\mapsto V_{p}\left(
x\right) $ defined on $[0,\infty )$ by%
\begin{equation*}
V_{p}\left( x\right) =\frac{1}{3p^{2}}\cosh px+1-\frac{1}{3p^{2}}\text{ if }%
p\neq 0\text{ and }V_{0}\left( x\right) =1+\frac{x^{2}}{6}
\end{equation*}%
is increasing.
\end{lemma}

\begin{proof}
Differentiation yields%
\begin{eqnarray*}
\frac{\partial V_{p}}{\partial p} &=&\frac{1}{3p^{3}}\left( px\sinh
px-2\cosh px+2\right) \\
&=&\frac{2x}{3p^{2}}\left( \cosh \frac{px}{2}-\frac{\sinh \frac{px}{2}}{%
\frac{px}{2}}\right) \sinh \frac{px}{2}>0,
\end{eqnarray*}%
which completes the proof.
\end{proof}

\begin{lemma}
\label{Lemma G_p}Let the function $G_{p}$ be defined on $(0,\infty )$ by%
\begin{equation*}
G_{p}\left( x\right) =\tfrac{\sinh x}{x}-(\tfrac{1}{3p^{2}}\cosh px+1-\tfrac{%
1}{3p^{2}})\text{ if }p\neq 0\text{ and }G_{0}\left( x\right) =\tfrac{\sinh x%
}{x}-1-\tfrac{x^{2}}{6}.
\end{equation*}

(i) If $G_{p}\left( x\right) <0$ for all $x\in (0,\infty )$, then $p\geq 1$.

(ii) If $G_{p}\left( x\right) >0$ for all $x\in (0,\pi /2)$, then $p\leq
p_{1}=\sqrt{15}/5\approx 0.77460$.
\end{lemma}

\begin{proof}
In order to prove the desired results, we need the following two relations:%
\begin{eqnarray}
\lim_{x\rightarrow 0}\frac{G_{p}\left( x\right) }{x^{4}} &=&-\frac{1}{360}%
\left( 5p^{2}-3\right) ,  \label{LimitG_p1} \\
\lim_{x\rightarrow \infty }\frac{G_{p}\left( x\right) }{e^{px}} &=&\left\{ 
\begin{array}{ll}
-\tfrac{1}{6p^{2}} & \text{if }p>1, \\ 
-\tfrac{1}{6} & \text{if }p=1, \\ 
\infty & \text{if }0<p<1, \\ 
\infty & \text{if }p=0.%
\end{array}%
\right.  \label{LimitG_p2}
\end{eqnarray}%
The first one follows by expanding in power series:%
\begin{equation*}
G_{p}\left( x\right) =-\frac{1}{360}\left( 5p^{2}-3\right) x^{4}+o\left(
x^{6}\right) .
\end{equation*}%
To obtain the second one, it needs to note that 
\begin{equation*}
e^{-px}G_{p}\left( x\right) =e^{(1-p)x}\tfrac{1-e^{-2x}}{2x}-\tfrac{1}{3p^{2}%
}\frac{1-e^{-2px}}{2}-\left( 1-\tfrac{1}{3p^{2}}\right) e^{-px},
\end{equation*}%
which gives (\ref{LimitG_p2}).

(i) If $G_{p}\left( x\right) <0$ for all $x\in (0,\infty )$, then we have $%
\lim_{x\rightarrow 0}x^{-4}G_{p}\left( x\right) \leq 0$ and $%
\lim_{x\rightarrow \infty }e^{-px}G_{p}\left( x\right) \leq 0$. These
together with (\ref{LimitG_p1}) and (\ref{LimitG_p2}) give $p\geq 1$.

(ii) If $G_{p}\left( x\right) >0$ for all $x\in (0,\infty )$, then we have $%
\lim_{x\rightarrow 0}x^{-4}G_{p}\left( x\right) \geq 0$ and $%
\lim_{x\rightarrow \infty }e^{-px}G_{p}\left( x\right) \geq 0$. These
together with (\ref{LimitG_p1}) and (\ref{LimitG_p2}) indicate $p\leq p_{1}$.
\end{proof}

We now can prove Theorem \ref{MTh}.

\begin{proof}[Proof of Theorem \protect\ref{MTh}]
The necessity follows by Lemma \ref{Lemma G_p}. To prove the sufficiency, we
expanding $G_{p}\left( x\right) $ in power series to get%
\begin{eqnarray*}
G_{p}\left( x\right) &=&\frac{\sinh x}{x}-(\tfrac{1}{3p^{2}}\cosh px+1-%
\tfrac{1}{3p^{2}}) \\
&=&\sum_{n=0}^{\infty }\frac{x^{2n}}{\left( 2n+1\right) !}-\left( \frac{1}{%
3p^{2}}\sum_{n=0}^{\infty }\frac{\left( px\right) ^{2n}}{\left( 2n\right) !}%
+1-\frac{1}{3p^{2}}\right) \\
&=&\sum_{n=2}^{\infty }\frac{3-\left( 2n+1\right) p^{2n-2}}{3\left(
2n+1\right) !}x^{2n}=\sum_{n=2}^{\infty }\frac{a_{n}\left( p^{2}\right) }{%
3\left( 2n+1\right) !}x^{2n}.
\end{eqnarray*}%
It is derived from Lemma \ref{Lemma a_n} that $a_{n}\left( p^{2}\right) \geq
0$ if $0<p\leq \sqrt{15}/5$, and clearly, $a_{n}\left( p^{2}\right) <0$ if $%
p\geq 1$.
\end{proof}

\section{Applications}

As simple applications of main results, we present some precise estimations
for certain special functions and constants in this section.

The sine integral is defined by%
\begin{equation*}
\func{Si}\left( t\right) =\int_{0}^{t}\frac{\sin x}{x}dx.
\end{equation*}%
Some estimates for sine integral can be seen \cite{Qi.12(4)(1996)}, \cite%
{Wu.19(12)(2006)}, \cite{Wu.12(2)(2008)}, \cite{Yang.GJM.2.1.2014}, \cite%
{Yang.JMI.7.4.2013}. Now we give a new result.

\begin{proposition}
For $t\in (0,\pi /2)$ and $p\in (0,\sqrt{15}/5]$, we have 
\begin{equation}
\tfrac{\sin pt}{3p^{3}}+\left( 1-\tfrac{1}{3p^{2}}\right) t+\tfrac{%
c_{0}\left( p\right) }{5}t^{5}<\limfunc{Si}\left( t\right) <\tfrac{\sin pt}{%
3p^{3}}+\left( 1-\tfrac{1}{3p^{2}}\right) t+\tfrac{c_{1}\left( p\right) }{5}%
t^{5},  \label{A1}
\end{equation}%
where $c_{0}\left( p\right) =\left( \pi /2\right) ^{-4}F_{p}\left( \pi
/2^{-}\right) $ and $c_{1}\left( p\right) =\left( 3-5p^{2}\right) /360$,
here $F_{p}\left( \pi /2^{-}\right) $ is defined by (\ref{F_p=0}).
Particularly, putting $p=0^{+},2/3$, we have%
\begin{eqnarray}
t-\frac{1}{18}t^{3}+\frac{2\pi ^{3}-48\pi +96}{15\pi ^{5}}t^{5} &<&\limfunc{%
Si}\left( t\right) <t-\frac{1}{18}t^{3}+\frac{1}{600}t^{5},  \label{A1p=0} \\
\frac{9}{8}\sin \frac{2t}{3}+\frac{1}{4}t+\frac{2\left( 16-5\pi \right) }{%
5\pi ^{5}}t^{5} &<&\limfunc{Si}\left( t\right) <\frac{9}{8}\sin \frac{2t}{3}+%
\frac{1}{4}t+\frac{7}{16200}t^{5},  \label{A1p=2/3}
\end{eqnarray}%
and then,%
\begin{eqnarray*}
1.3705 &\approx &\frac{2}{5}\pi -\frac{1}{360}\pi ^{3}+\frac{1}{5}<\limfunc{%
Si}\left( \tfrac{\pi }{2}\right) <\frac{1}{2}\pi -\frac{1}{144}\pi ^{3}+%
\frac{1}{19\,200}\pi ^{5}\approx 1.3714, \\
1.3706 &\approx &\frac{1}{16}\pi +\frac{9\sqrt{3}}{16}+\frac{1}{5}<\limfunc{%
Si}\left( \tfrac{\pi }{2}\right) <\frac{1}{8}\pi +\frac{7}{518\,400}\pi ^{5}+%
\frac{9\sqrt{3}}{16}\approx 1.3711.
\end{eqnarray*}
\end{proposition}

\begin{proof}
Integrating each sides in (\ref{sin E}) over $\left[ 0,t\right] $ yields (%
\ref{A1}). Taking the limits of the left and right hand sides in\ (\ref{A1})
as $p\rightarrow 0^{+}$ gives (\ref{A1p=0}), and putting $p=2/3$ in (\ref{A1}%
) leads to (\ref{A1p=2/3}). Substituting $t=\pi /2$ into (\ref{A1p=0}) and (%
\ref{A1p=2/3}) we get the last two approximations of $\limfunc{Si}\left( \pi
/2\right) $.
\end{proof}

It is known that%
\begin{equation*}
\int_{0}^{\infty }\frac{x}{\sinh x}dx=\frac{1}{2}\psi ^{\prime }(\tfrac{1}{2}%
)=\frac{\pi ^{2}}{4},
\end{equation*}%
where $\psi ^{\prime }$ is the tri-gamma function defined by%
\begin{equation*}
\psi ^{\prime }(t)=\int_{0}^{\infty }\frac{xe^{-tx}}{1-e^{-x}}dx.
\end{equation*}%
We define 
\begin{equation*}
Sh\left( t\right) =\int_{0}^{t}\frac{x}{\sinh x}dx.
\end{equation*}%
Then by (\ref{M2}) we have%
\begin{equation*}
\frac{3}{\cosh x+2}<\frac{x}{\sinh x}<\frac{9}{5\cosh (\sqrt{15}x/5)+4}.
\end{equation*}%
Integrating over $\left[ 0,t\right] $ and calculating lead to

\begin{proposition}
For $t>0$, we have%
\begin{eqnarray}
&&\sqrt{3}\ln \tfrac{e^{t}-\sqrt{3}+2}{e^{t}+\sqrt{3}+2}-\sqrt{3}\ln \left(
2-\sqrt{3}\right)  \notag \\
&<&Sh\left( t\right) <2\sqrt{15}\arctan \left( \tfrac{5}{3}e^{\sqrt{15}t/5}+%
\tfrac{4}{3}\right) -2\sqrt{15}\arctan 3.  \label{A2}
\end{eqnarray}%
In particular, we have%
\begin{equation*}
4.5621\approx 2\sqrt{3}\ln \left( 2+\sqrt{3}\right) <\psi ^{\prime }(\tfrac{1%
}{2})<2\sqrt{15}\pi -4\sqrt{15}\arctan 3\approx 4.9845.
\end{equation*}
\end{proposition}

The Catalan constant \cite{Catalan.31(7)}%
\begin{equation*}
G=\sum_{n=0}^{\infty }\frac{\left( -1\right) ^{n}}{\left( 2n+1\right) ^{2}}%
=0.9159655941772190...
\end{equation*}%
is a famous mysterious constant appearing in many places in mathematics and
physics. Its integral representations contain the following \cite%
{Bradley.2001}%
\begin{equation}
G=\int_{0}^{1}\frac{\arctan x}{x}dx=\frac{1}{2}\int_{0}^{\pi /2}\frac{x}{%
\sin x}dx=\frac{\pi ^{2}}{16}-\frac{\pi }{4}\ln 2+\int_{0}^{\pi /4}\frac{%
x^{2}}{\sin ^{2}x}dx.  \label{G}
\end{equation}%
We present an estimation for $G$ below.

\begin{proposition}
We have 
\begin{equation}
0.91586\approx \tfrac{\sqrt{15}}{4}\ln \tfrac{4\cos \frac{\sqrt{15}\pi }{10}%
+3\sin \frac{\sqrt{15}\pi }{10}+5}{4\cos \frac{\sqrt{15}\pi }{10}-3\sin 
\frac{\sqrt{15}\pi }{10}+5}<G<\tfrac{\sqrt{15}}{5}\ln \tfrac{11\sqrt{2-\sqrt{%
2}}+3\sqrt{15}\sqrt{\sqrt{2}+2}+32}{11\sqrt{2-\sqrt{2}}-3\sqrt{15}\sqrt{%
\sqrt{2}+2}+32}\approx 0.91675,  \label{A3}
\end{equation}
\end{proposition}

\begin{proof}
From the fourth and fifth inequalities in (\ref{M1c}) we obtain that for $%
x\in (0,\pi /2)$, the two-side inequality 
\begin{equation*}
\frac{1}{\tfrac{5}{9}\cos \tfrac{\sqrt{15}x}{5}+\tfrac{4}{9}}<\frac{x}{\sin x%
}<\frac{1}{\tfrac{16}{27}\cos \tfrac{3x}{4}+\tfrac{11}{27}}
\end{equation*}%
holds true. Integrating both sides over $\left[ 0,\pi /2\right] $ yields%
\begin{equation*}
\int_{0}^{\pi /2}\frac{dx}{\tfrac{5}{9}\cos \tfrac{\sqrt{15}x}{5}+\tfrac{4}{9%
}}<\int_{0}^{\pi /2}\frac{x}{\sin x}dx<\int_{0}^{\pi /2}\frac{dx}{\tfrac{16}{%
27}\cos \tfrac{3x}{4}+\tfrac{11}{27}}.
\end{equation*}%
Direct computations give%
\begin{eqnarray*}
\int_{0}^{\pi /2}\frac{dx}{\tfrac{5}{9}\cos \tfrac{\sqrt{15}x}{5}+\tfrac{4}{9%
}} &=&\frac{\sqrt{15}}{2}\ln \tfrac{4\cos \frac{\sqrt{15}\pi }{10}+3\sin 
\frac{\sqrt{15}\pi }{10}+5}{4\cos \frac{\sqrt{15}\pi }{10}-3\sin \frac{\sqrt{%
15}\pi }{10}+5}\approx 1.8317, \\
\int_{0}^{\pi /2}\frac{dx}{\tfrac{16}{27}\cos \tfrac{3x}{4}+\tfrac{11}{27}}
&=&\frac{2\sqrt{15}}{5}\ln \tfrac{11\sqrt{2-\sqrt{2}}+3\sqrt{15}\sqrt{\sqrt{2%
}+2}+32}{11\sqrt{2-\sqrt{2}}-3\sqrt{15}\sqrt{\sqrt{2}+2}+32}\approx 1.8335
\end{eqnarray*}%
Utilizing the the second formula in (\ref{G}) (\ref{A2}) follows.
\end{proof}

We close this paper by giving some inequalities for bivariate means.

The Schwab-Borchardt mean of two numbers $a\geq 0$ and $b>0$, denoted by $%
SB\left( a,b\right) $, is defined as \cite[Theorem 8.4]{Biernacki.9.1955}, 
\cite[3, (2.3)]{Carlson.78(1971)}%
\begin{equation*}
SB\left( a,b\right) =\left\{ 
\begin{array}{cc}
\frac{\sqrt{b^{2}-a^{2}}}{\arccos \left( a/b\right) } & \text{if }a<b, \\ 
a & \text{if }a=b, \\ 
\frac{\sqrt{a^{2}-b^{2}}}{\func{arccosh}\left( a/b\right) } & \text{if }a>b.%
\end{array}%
\right.
\end{equation*}%
The properties and certain inequalities involving Schwab-Borchardt mean can
be found in \cite{Neuman.14(2003)}, \cite{Neuman.JMI.5.4.2011}. We now
establish a new inequality for this mean.

For $a<b$, letting $x=\arccos \left( a/b\right) $ in the fourth inequality
of (\ref{M1c}) and using half-angle and triple-angle formulas for cosine
function, and multiplying two sides by $b$, we get%
\begin{eqnarray*}
SB\left( a,b\right) &\geq &\tfrac{16}{27}b\sqrt{\frac{1+\sqrt{\frac{%
1+4\left( a/b\right) ^{3}-3a/b}{2}}}{2}}+\tfrac{11}{27}b \\
&=&\tfrac{8}{27}\left( \sqrt{2\left( b-2a\right) ^{2}\left( a+b\right) }%
+2b^{3/2}\right) ^{1/2}b^{1/4}+\tfrac{11}{27}b.
\end{eqnarray*}%
For $a>b$, letting $x=\func{arccosh}\left( a/b\right) $ in the inequality
connecting the fourth and sixth members of (\ref{M1c}) and using half-angle
and triple-angle formulas for hyperbolic cosine function, and multiplying
two sides by $b$, we get the same inequality as above.

\begin{proposition}
For $a,b>0$, we have%
\begin{equation}
SB\left( a,b\right) \geq \tfrac{8\sqrt{2}}{27}\left( |b-2a|\sqrt{\frac{a+b}{2%
}}+b^{3/2}\right) ^{1/2}b^{1/4}+\tfrac{11}{27}b.  \label{A4}
\end{equation}
\end{proposition}

\begin{remark}
From the inequality (\ref{A4}), it is easy to get%
\begin{equation*}
SB\left( a,b\right) \geq \frac{11+8\sqrt{2}}{27}b\approx 0.82643\times b
\end{equation*}
due to $|b-2a|\geq 0$. It seems to new and interesting.
\end{remark}

For $a,b>0$, with $x=\left( 1/2\right) \ln \left( a/b\right) $, we have%
\begin{equation*}
\frac{\sinh x}{x}=\frac{L(a,b)}{G\left( a,b\right) }\text{, \ \ \ }\cosh px=%
\frac{\left( a^{p}+b^{p}\right) /2}{\left( \sqrt{ab}\right) ^{p}}=\frac{%
A_{p}^{p}\left( a,b\right) }{G^{p}\left( a,b\right) },
\end{equation*}%
and by Theorem \ref{MTh} we immediately get the following

\begin{proposition}
For $a,b>0$ with $a\neq b$, the double inequality%
\begin{equation}
\tfrac{5}{9}A_{\sqrt{15}/5}^{\sqrt{15}/5}G^{1-\sqrt{15}/5}+\tfrac{4}{9}G<L<%
\tfrac{1}{3}A+\tfrac{2}{3}G  \label{A5}
\end{equation}%
holds with the best constants $p_{1}=\sqrt{15}/5$ and $1$. And, the function%
\begin{equation*}
p\mapsto \tfrac{1}{3p^{2}}A_{p}^{p}G^{1-p}+\left( 1-\tfrac{1}{3p^{2}}\right)
G\text{ if }p\neq 0\text{ and }G+\frac{\left( \ln b-\ln a\right) ^{2}}{24}G%
\text{ if }p=0
\end{equation*}%
is increasing on $\mathbb{R}$.
\end{proposition}

\begin{remark}
Accordingly, Corollary \ref{MThc} can be changed into the chain of
inequalities for means:%
\begin{eqnarray*}
A_{1/\sqrt{3}}^{1/\sqrt{3}}G^{1-1/\sqrt{3}} &<&\tfrac{3}{4}%
A_{2/3}^{2/3}G^{1/3}+\tfrac{1}{4}G<\tfrac{2}{3}A_{1/\sqrt{2}}^{1/\sqrt{2}%
}G^{1-1/\sqrt{2}}+\tfrac{1}{3}G \\
&<&\tfrac{16}{27}A_{3/4}^{3/4}G^{1/4}+\tfrac{11}{27}G<\tfrac{5}{9}A_{\sqrt{15%
}/5}^{\sqrt{15}/5}G^{1-\sqrt{15}/5}+\tfrac{4}{9}G<L \\
&<&\tfrac{1}{3}A+\tfrac{2}{3}G<\tfrac{1}{2}A_{1/\sqrt{3}}^{2/\sqrt{3}}G^{1-2/%
\sqrt{3}}+\tfrac{1}{2}G.
\end{eqnarray*}
\end{remark}

\begin{remark}
In \cite{Yang.JIA.2013.116}, Yang obtained a sharp lower bound $%
A_{q_{0}}^{1/\left( 3q_{0}\right) }G^{1-1/\left( 3q_{0}\right) }$ for the
logarithmic mean $L$, where $q_{0}=1/\sqrt{5}$, and pointed out that this
one seems to superior to most of known ones. Now, we derive a new sharp
lower bound $\tfrac{5}{9}A_{p_{1}}^{p_{1}}G^{1-p_{1}}+\tfrac{4}{9}G$ for $L$%
, where $p_{1}=\sqrt{15}/5$. We claim that the latter is better than the
former. In fact, we have%
\begin{equation}
L>\tfrac{5}{9}A_{p_{1}}^{p_{1}}G^{1-p_{1}}+\tfrac{4}{9}G>A_{q}^{1/\left(
3q\right) }G^{1-1/\left( 3q\right) }  \label{L comp.}
\end{equation}%
if and only if $q\geq q_{0}=1/\sqrt{5}$. In order for the second inequality
in (\ref{L comp.}) to hold, it suffices that for $x>0$%
\begin{equation*}
D\left( x\right) =\tfrac{1}{3p_{1}^{2}}\cosh p_{1}x+1-\tfrac{1}{3p_{1}^{2}}%
-\left( \cosh qx\right) ^{1/\left( 3q^{2}\right) }>0
\end{equation*}%
if and only if $q\geq q_{0}=1/\sqrt{5}$.

The necessity can be obtained by $\lim_{x\rightarrow 0}x^{-4}D\left(
x\right) \geq 0$, which follows by expanding in power series%
\begin{equation*}
D\left( x\right) =\frac{1}{72}x^{4}\left( p_{1}^{2}+2q^{2}-1\right) +o\left(
x^{6}\right) .
\end{equation*}%
Then, $q\geq \sqrt{\left( 1-p_{1}^{2}\right) /2}=1/\sqrt{5}$.

Since\ $q\mapsto \left( \cosh qx\right) ^{1/\left( 3q^{2}\right) }$ is
decreasing on $\left( 0,\infty \right) $ proved in \cite[Lemma 2]%
{Yang.JIA.2013.116}, to prove $D\left( x\right) \geq 0$ if $q\geq q_{0}$, it
suffices to show that $D\left( x\right) \geq 0$ when $q=q_{0}$.
Differentiation yields%
\begin{eqnarray*}
D^{\prime }\left( x\right) &=&\frac{1}{3p_{1}}\sinh p_{1}x-\frac{1}{3q_{0}}%
\cosh ^{1/\left( 3q_{0}^{2}\right) -1}q_{0}x\sinh q_{0}x \\
&=&\frac{\sinh q_{0}x}{3q_{0}}\left( \frac{q_{0}}{p_{1}}\frac{\sinh p_{1}x}{%
\sinh q_{0}x}-\cosh ^{1/\left( 3q_{0}^{2}\right) -1}q_{0}x\right) \\
&=&\frac{\sinh q_{0}x}{3q_{0}}\times L\left( \frac{q_{0}}{p_{1}}\frac{\sinh
p_{1}x}{\sinh q_{0}x},\cosh ^{1/\left( 3q_{0}^{2}\right) -1}q_{0}x\right)
\times D_{1}\left( x\right) ,
\end{eqnarray*}%
where%
\begin{equation*}
D_{1}\left( x\right) =\ln \left( \frac{q_{0}}{p_{1}}\frac{\sinh p_{1}x}{%
\sinh q_{0}x}\right) -\left( \frac{1}{3q_{0}^{2}}-1\right) \ln \cosh q_{0}x.
\end{equation*}%
Differentiating $D_{1}\left( x\right) $ gives%
\begin{equation*}
D_{1}^{\prime }\left( x\right) =\frac{D_{2}\left( x\right) }{6q_{0}\sinh
2q_{0}x\sinh p_{1}x},
\end{equation*}%
where%
\begin{eqnarray*}
D_{2}\left( x\right) &=&4(-\sinh p_{1}x\sinh ^{2}q_{0}x-3q_{0}^{2}\cosh
^{2}q_{0}x\sinh p_{1}x \\
&&+3q_{0}^{2}\sinh p_{1}x\sinh ^{2}q_{0}x+3p_{1}q_{0}\cosh p_{1}x\cosh
q_{0}x\sinh q_{0}x).
\end{eqnarray*}%
Utilizing "product into sum" formulas and expanding in power series lead to%
\begin{eqnarray*}
D_{2}\left( x\right) &=&-2\left( 6q_{0}^{2}-1\right) \sinh p_{1}x+\left(
3p_{1}q_{0}-1\right) \sinh \left( p_{1}x+2q_{0}x\right) \\
&&+\left( 3p_{1}q_{0}+1\right) \sinh \left( 2q_{0}x-p_{1}x\right) \\
&=&\sum_{n=1}^{\infty }d_{n}\tfrac{p_{1}^{2n-1}x^{2n-1}}{\left( 2n-1\right) !%
},
\end{eqnarray*}%
where%
\begin{equation*}
d_{n}=\left( 3p_{1}q_{0}-1\right) \left( 1+\tfrac{2q_{0}}{p_{1}}\right)
^{2n-1}+\left( 3p_{1}q_{0}+1\right) \left( \tfrac{2q_{0}}{p_{1}}-1\right)
^{2n-1}-2\left( 6q_{0}^{2}-1\right)
\end{equation*}%
Now we show that $d_{n}\geq 0$ for $n\geq 1$. A simple verification yields $%
d_{1}=d_{2}=0$, $d_{3}=64/45>0$. Suppose that $d_{n}>0$ for $n>3$, that is,%
\begin{equation*}
\left( 3p_{1}q_{0}+1\right) \left( \frac{2q_{0}}{p_{1}}-1\right)
^{2n-1}>2\left( 6q_{0}^{2}-1\right) -\left( 3p_{1}q_{0}-1\right) \left( 1+%
\frac{2q_{0}}{p_{1}}\right) ^{2n-1}.
\end{equation*}%
Then,%
\begin{eqnarray*}
d_{n+1} &=&\left( 3p_{1}q_{0}-1\right) \left( 1+\tfrac{2q_{0}}{p_{1}}\right)
^{2n+1}+\left( 3p_{1}q_{0}+1\right) \left( \tfrac{2q_{0}}{p_{1}}-1\right)
^{2n+1}-2\left( 6q_{0}^{2}-1\right) \\
&>&\left( 3p_{1}q_{0}-1\right) \left( 1+\tfrac{2q_{0}}{p_{1}}\right)
^{2n+1}+\left( 2\left( 6q_{0}^{2}-1\right) -\left( 3p_{1}q_{0}-1\right)
\left( 1+\tfrac{2q_{0}}{p_{1}}\right) ^{2n-1}\right) \\
&&\times \left( \frac{2q_{0}}{p_{1}}-1\right) ^{2}-2\left(
6q_{0}^{2}-1\right) \\
&=&\frac{8q_{0}}{p_{1}^{2}}\left( p_{1}\left( 3p_{1}q_{0}-1\right) \left( 1+%
\frac{2q_{0}}{p_{1}}\right) ^{2n-1}-\left( 6q_{0}^{2}-1\right) \left(
p_{1}-q_{0}\right) \right) .
\end{eqnarray*}%
Since $\left( 3p_{1}q_{0}-1\right) =\left( 3\sqrt{3}-5\right) /5>0$, using
binomial expansion we get%
\begin{eqnarray*}
\frac{p_{1}^{2}}{8q_{0}}d_{n+1} &>&p_{1}\left( 3p_{1}q_{0}-1\right) \left(
1+\left( 2n-1\right) \frac{2q_{0}}{p_{1}}\right) -\left( 6q_{0}^{2}-1\right)
\left( p_{1}-q_{0}\right) \\
&=&2q_{0}\left( 3p_{1}q_{0}-1\right) \left( 2n-1\right) +q_{0}\left(
3p_{1}^{2}+6q_{0}^{2}-6p_{1}q_{0}-1\right) \\
&=&2q_{0}\left( 3p_{1}q_{0}-1\right) \left( 2n-1\right) -2q_{0}\left(
3p_{1}q_{0}-1\right) \\
&=&4q_{0}\left( 3p_{1}q_{0}-1\right) \left( n-1\right) >0,
\end{eqnarray*}%
where the third equality holds due to $3p_{1}^{2}+6q_{0}^{2}=3$. By
mathematical induction, we have proven $D_{2}\left( x\right) \geq 0$ for $%
n\geq 1$. It follows that $D_{1}^{\prime }\left( x\right) >0$, which means
that $D_{1}$ is increasing on $\left( 0,\infty \right) $, and then $%
D_{1}\left( x\right) >\lim_{x\rightarrow 0^{+}}D\left( x\right) =0$. This in
turn implies that $D$ is increasing on $\left( 0,\infty \right) $, and
therefore, $D\left( x\right) \geq D\left( 0^{+}\right) =0$, which proves the
sufficiency.
\end{remark}

\end{document}